\documentclass[10pt]{amsart}
\usepackage[linktocpage=true,colorlinks=true,citecolor=blue]{hyperref}
\usepackage{amsrefs, amssymb,amsfonts,amsthm}
\usepackage{enumerate}
\usepackage[applemac]{inputenc}
\usepackage[OT2,T1]{fontenc}

\usepackage{wrapfig}

\usepackage{url}

\DefineSimpleKey{bib}{myurl}

\newcommand\myurl[1]{\url{#1}}

\BibSpec{webpage}{%
  +{}{\PrintAuthors} {author}
  +{,}{ \textit} {title}
  +{}{ \parenthesize} {date}
  +{,}{ \myurl} {myurl}
  +{,}{ } {note}
  +{.}{ } {transition}
}

\usepackage{graphicx}
\usepackage{float}

\newtheorem{thm}{Theorem}

\newtheorem{prop}[thm]{Proposition}
\newtheorem{lem}{Lemma}
\newtheorem{conj}[thm]{Conjecture}
\newtheorem{quest}[thm]{Question}

\theoremstyle{definition}
\newtheorem{defn}[thm]{Definition}

\theoremstyle{remark}
\newtheorem{rem}[thm]{Remark}


\DeclareMathOperator\up{up}

\def\nn{\mskip12mu}

\title{Arithmetic of the Fabius function}
\author{J. Arias de Reyna}
\address{Univ.~de Sevilla \\
Facultad de Matem\'aticas \\
c/ Tarfia, sn
 \\
41012-Sevilla \\
Spain} 
\email{arias@us.es}

\date{\today}

\begin{document}

\newcommand{\N}{{\mathbb N}}
\newcommand{\R}{{\mathbb R}}
\newcommand{\C}{{\mathbb C}}
\newcommand{\Z}{{\mathbb Z}}
\newcommand{\Q}{{\mathbb Q}}
\newcommand{\arctanh}{\mathop{\rm arctanh }}
\def\Re{\operatorname{Re}}
\def\Im{\operatorname{Im}}
\def\Hermite{\operatorname{\mathcal H}}
\def\supp{\operatorname{supp}}
\def\den{\operatorname{den}}

\newfont{\cmbsy}{cmbsy10}
\newcommand{\Orden}{\mathop{\hbox{\cmbsy O}}\nolimits}

\begin{abstract} 
The Fabius function $F(x)$ takes rational values at dyadic points. We study the 
arithmetic of these rational numbers. In particular, we define two sequences
of natural numbers that determine these rational numbers. Using these sequences 
we solve 
a conjecture raised  in MathOverflow by Vladimir Reshetnikov.
We determine the dyadic valuation of $F(2^{-n})$, showing that 
$\nu_2(F(2^{-n}))=-\binom{n}{2}-\nu_2(n!)-1$.
We give the proof  (Proposition \ref{P:Res}) of  a formula that allows an efficient computation 
of exact or approximate values of $F(x)$.  

The Fabius function was defined in 1935 by 
Jessen and Wintner and has been independently defined  at least 
six times since. 
We attempt to unify notations related to the Fabius function. 
\end{abstract}

\maketitle

\section{Introduction.}
The Fabius function is a natural object. It is not surprising that it has been
independently  defined several times. As far as known it was defined independently 
in the following cases:
\begin{enumerate}
\item In 1935 by  B. Jessen and A. Wintner  \cite{MR1501802}*{Ex.~5, p.~62} (English). In five lines
they showed that it is infinitely differentiable.

\item In 1966 by J. Fabius \cite{MR0197656} (English); considered  as the distribution function 
of a random variable. 

\item In 1971 by  V. A.~Rvach{\"e}v   \cite{MR0296237} (Ukrainian); defined 
as a solution to the equation
$y'(x)=2(y(2x+1)-y(2x-1))$.

\item In 1981 by G. Kh. Kirov and G. A.~Totkov \cite{MR705073} (Bulgarian).

\item In 1982 by J. Arias de Reyna \cite{MR666839} (Spanish); defined  in the same way as V. A.~Rvach{\"e}v. 

\item In 1985  by R. Schnabl \cite{MR1031820} (German). 
\end{enumerate}

Today there are many papers dealing with the Fabius function, due mainly to {V. A.~Rvach{\"e}v} \cite{MR0296237}--\cite{MR1050928}, (a 
summary of his result can be found in  \cite{MR1050928}). Recently   \cite{MR666839}
has been translated to English and posted  in 
\href{http://arxiv.org/abs/1702.05442}{arXiv:1702.05442}. In this paper  the function is 
defined and its main properties are proven. 

The function $\up(x)$ takes rational values at dyadic points \cite{MR0330388}.%
\footnote{This Russian paper was not accessed directly. Rather, it came to notice
through the works of  V. A.~Rvach{\"e}v \cite{MR1050928}. In \cite{MR666839} it was
independently proved that the values of $\up(k/2^n)$ are rational numbers.}
The object of this paper are
the arithmetical properties of these rational values.  Some of the results 
are also contained in \cite{MR666839}. In particular, we solve here a question 
posed in MathOverflow by Vladimir Reshetnikov.

\subsection{Some notations.} 
We consider  two functions. Firstly, the Rvach{\"e}v function $\up(x)$ with support in $[-1,1]$ that satisfies on all
$\R$ equation \eqref{E:main}.  V. A.~Rvach{\"e}v denotes it by $\up(x)$; it is the same function 
that is called
$\varphi(x)$ in the translation of \cite{MR666839}. Secondly, the Fabius  function $F(x)$, 
defined by Fabius, that satisfies $F'(x)=2F(2x)$  for $x>0$. This function coincides with 
the function $\theta(x)$ in \cite{MR666839}, except that  in \cite{MR666839}
it is defined on all of $\R$, putting $F(x)=0$ for $x<0$. Therefore $F'(x)=2F(2x)$  is true 
for all $x\in\R$. 

There is a connection between these two functions (see \cite{MR666839}*{Theorem 4}). 
Specifically, 
\begin{equation}\label{E, F-def}
F(x)=\sum_{n=0}^\infty (-1)^{w(n)}\up(x-2n-1),
\end{equation}
where $w(n)$ is the sum of the digits of $n$ expressed in base $2$. 

\begin{figure}[htbp]
\begin{center}
\includegraphics[width=\textwidth]{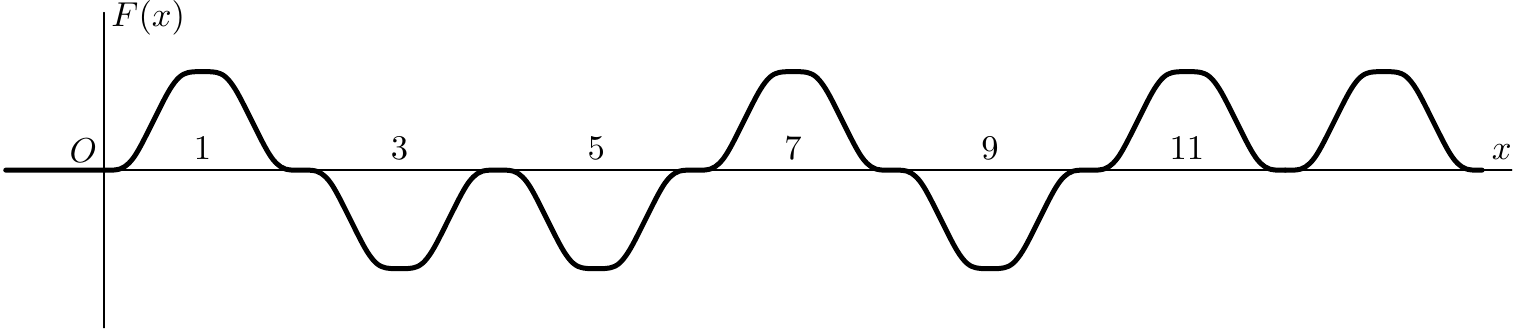}
\label{fabius}
\end{center}
\end{figure}

The Fabius function $F\colon[0,1]\to[0,1]$ is increasing with $F(0)=0$ and $F(1)=1$. 
Rvach{\"e}v's function $\up\colon\R\to\R$ can be defined in terms of $F(t)$ by 
\[\up(t)=\begin{cases} F(t+1) & \text{for $-1\le t\le0$},\\
F(1-t) & \text{for $0\le t\le 1$},\\
0 & \text{for $|t|>1$}.
\end{cases}
\]
The function $\up(t)$ is an infinitely differentiable function of compact 
support $[-1,1]$. 
The relation with the Fabius function is only given for reference. 
The function $\up(t)$ has a nice independent
definition, see Theorem 1 in \cite{MR666839}. 

We use $w(n)$ to denote the sum of the digits of $n$ expressed in binary notation. 
We use $\nu_2(r)$ is the dyadic  valuation of the rational number $r$, the exponent of $2$ in the 
factorization of $r$. 

We will define two sequences of rational numbers. Firstly,  $(c_n)$ the moments of Rvach{\"e}v's function $\up(x)$. Secondly, 
$(d_n)$ the half moments of Rvach{\"e}v's function. These are also related to the values
of the Fabius 
function $F(x)$, see equation \eqref{E:dnint}.  There are two sequences of natural numbers associated to these rational 
numbers: $F_n$ associated to the moments $c_n$,  and $G_n$ associated to the half moments. 

A further  two sequences are Reshetnikov's numbers $R_n$, and the sequence $D_n$ of
the  common denominators of
the values  $F(a/2^n)$ of the Fabius function at dyadic numbers with denominator $2^n$.

\section{Properties of the Fabius function.}

The Fabius function $\up(x)$ is a simple
infinitely differentiable function of compact support. Visualising it as a bump function,
it   increases from $0$ to a certain value and then decreases symmetrically 
to $0$. Its derivative is positive in the first interval and then negative in the next.
Therefore it may be  that the derivative is obtained from two 
homothetically deformed copies of the function. 

\begin{wrapfigure}{r}{0.40\textwidth}
\includegraphics[width=0.38\textwidth]{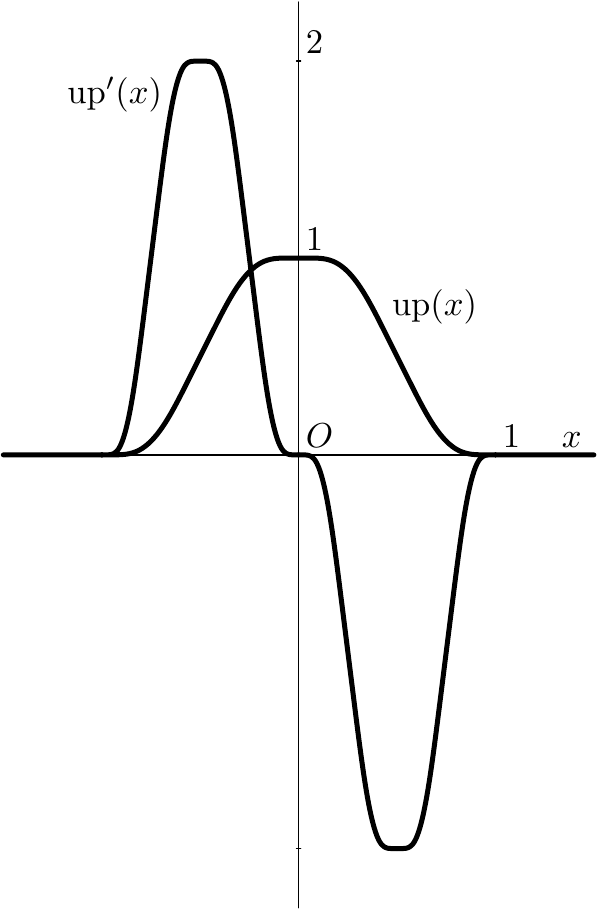}
\end{wrapfigure} 
\leavevmode%
In fact $\up(x)$ is infinitely differentiable of 
compact support equal to $[-1,1]$, and  for all $x\in\R$ satisfies  the equation

\begin{equation}\label{E:main}
\mskip-200mu \up'(x)=2\bigl(\up(2x+1)-\up(2x-1)\bigr),
\end{equation}
with $\up(0)=1$.

The number $2$ is the only factor giving an interesting solution to the equation.
The Fabius function \eqref{E, F-def} appears when one attempts to compute the 
derivatives $\up^{(n)}(x)$.
We have  $F'(x)=2F(2x)$ for all $x\in\R$. 
Therefore  (\cite{MR666839}*{eq.~(22)})
\begin{equation}\label{E:FabiusD}
\mskip-220mu 
F^{(k)}(x)=2^{\binom{k+1}{2}}F(2^k x),
\end{equation}
and
\begin{equation}
\mskip-220mu 
\up^{(n)}(x)=2^{\binom{n+1}{2}}F(2^n(x+1)), \quad -1\le x\le 1.
\end{equation}
Its Fourier transform is an entire function
\begin{equation}
\widehat{\up}(z)=\prod_{n=0}^\infty \frac{\sin(\pi z/2^n)}
{\pi z/2^n}.
\end{equation}
The inversion formula yields the following expression
\begin{equation}
\up(x)=\int_\R\widehat{\up}(t)e^{2\pi i t x}\,dt.
\end{equation}

\section{The values of the Fabius function and two sequences of rational numbers.}

\subsection{The sequence of rational  numbers $c_n$.}

The Fourier transform  $\widehat\up(z)$ is an entire function. 
In \cite{MR666839}*{eq.~(4) and (6)}  it is shown that  
\begin{equation}\label{hatphi}
\widehat\up(z)=\prod_{n=0}^\infty \frac{\sin(\pi z/2^n)}{\pi z/2^n}=\sum_{n=0}^\infty (-1)^n \frac{c_n}{(2n)!} (2\pi z)^{2n}.
\end{equation}
These coefficients $c_n$ are very important for the study of the Fabius function.
They are given by the integrals \cite{MR666839}*{eq.~(34)}
\begin{equation}\label{D:cn}
c_n=\int_\R t^{2n}\up(t)\,dt,
\end{equation}
and satisfy the recurrence \cite{MR666839}*{eq.~(7)}
\begin{equation}\label{E:cr}
(2n+1)2^{2n}c_n=\sum_{k=0}^{n}\binom{2n+1}{2k}c_k.
\end{equation}

The next Proposition is contained in \cite{MR666839} but is a little 
cryptic, so it is  expanded here.
\begin{prop}\label{PD:F}
The numbers $c_n$ satisfy the recurrence
\begin{equation}\label{E:cr2}
c_0=1,\quad (2n+1)(2^{2n}-1)c_n=\sum_{k=0}^{n-1}\binom{2n+1}{2k}c_k.
\end{equation}
For any natural number $n$ we have
\begin{equation}\label{E:cf}
c_n=\frac{F_n}{(2n+1)!!}\prod_{\nu=1}^n(2^{2\nu}-1)^{-1},
\end{equation}
where $F_n$ are natural numbers.
\end{prop}

\begin{proof}
The last term in the sum of \eqref{E:cr} is equal to $(2n+1)c_n$. We obtain 
\eqref{E:cr2} from \eqref{E:cr} by adding $-(2n+1)c_n$ to both sides. 
We have $c_0=\widehat\up(0)$ and this is equal to $1$.  Then the recurrence 
can be used to compute the numbers:

\[1,\nn\frac{1}{9},\nn\frac{19}{675},\nn\frac{583}{59\,535},
\nn\frac{132\,809}{32\,531\,625},
\nn\frac{46\,840\,699}{24\,405\,225\,075},\nn\frac{4\,068\,990\,560\,161}{4\,133\,856\,862\,760\,625},\dots\]
The recurrence shows that $c_n\ge0$ for all $n$.

Equation \eqref{E:cf} follows from the definition
\begin{equation}\label{E:Fnum_def}
F_n:= c_n(2n+1)!!\,\prod_{\nu=1}^n (2^{2\nu}-1).
\end{equation}
It remains to show that the numbers $F_n$ are natural.
We understand that an empty product is equal to $1$, so that $F_0=1$. For 
$n\ge1$  substituting \eqref{E:cf} in the recurrence \eqref{E:cr2} yields, for $n\ge1$,
\[\frac{F_n}{(2n-1)!!}\prod_{\nu=1}^{n-1}(2^{2\nu}-1)^{-1}=
\sum_{k=0}^{n-1}\binom{2n+1}{2k}\frac{F_k}{(2k+1)!!}\prod_{\nu=1}^{k}(2^{2\nu}-1)^{-1},
\]
so that
\begin{equation}\label{Rec:F}
F_n=\sum_{k=0}^{n-1}F_k\binom{2n+1}{2k}\frac{(2n-1)!!}{(2k+1)!!}
\prod_{\nu=k+1}^{n-1}(2^{2\nu}-1),\qquad n\ge1.
\end{equation}
This expression for $F_n$ shows by induction  that all the $F_n$ are natural numbers. 
The first $F_n$'s are
\[1,\ 1,\  19,\  2\,915,\  2\,788\,989,\  14\,754\,820\,185,\  402\,830\,065\,455\,939, \
54\,259\,734\,183\,964\,303\,995,\dots\]
\end{proof}

\subsection{Relation to  values of the Fabius function.}
Notice that $F(2^{-n})=\up(1-2^{-n})$. 
In \cite{MR666839}*{Theorem 6} it is shown that 
\begin{equation}
n\int_0^1x^{n-1}\up(x)\,dx=n!\,2^{\binom{n}{2}}\up(1-2^{-n}).
\end{equation}
Since $\up(x)$ is an even function, it follows that 
\begin{equation}
\frac{c_n}{2}=\int_0^1x^{2n}\up(x)\,dx=(2n)!\,2^{\binom{2n+1}{2}}
\up(1-2^{-2n-1}).
\end{equation}
Therefore, as noticed in \cite{MR666839} we have
\begin{equation}\label{E:35}
F(2^{-2n-1})=\up(1-2^{-2n-1})=\frac{2^{-\binom{2n+1}{2}}}{2 (2n)!}\frac{F_n}{(2n+1)(2n-1)\cdots1}
\prod_{k=1}^n(2^{2k}-1)^{-1}.
\end{equation}
The values of $F(2^{-n})=\up(1-2^{-n})$ depend of another interesting sequence of 
rational numbers.

\subsection{The sequence of rational  numbers $d_n$.}
The sequence of numbers $d_n$ is introduced in  \cite{MR666839}*{eq.(38) and (39)}
as the coefficients of the power series
\begin{equation}\label{E:dexpan}
f(x)=\sum_{n=0}^\infty \frac{d_n}{n!}x^n
\end{equation}
satisfying 
\begin{equation}\label{E:deff}
f(2x)=\frac{e^x-1}{x}f(x).
\end{equation}
So these numbers are defined also by the recurrence \cite{MR666839}*{eq.~(40)}
\begin{equation}\label{E:d-recur}
d_0=1,\quad (n+1)(2^n-1)d_n=\sum_{k=0}^{n-1}\binom{n+1}{k}d_k.
\end{equation}
The $d_n$ are connected with $\up(t)$ in  \cite{MR666839}*{eq.~(36)}.  
It is shown there that the function
\begin{equation}\label{E:gooddef}
f(x):=1+x\int_0^1\up(t)e^{xt}\,dt=e^{\frac{x}{2}}\widehat\up
\Bigl(\frac{ix}{4\pi}\Bigr)
\end{equation}
satisfies \eqref{E:deff}. We give the simple proof that was omitted in 
\cite{MR666839}.

\begin{prop}
Let $f$ be defined as in \eqref{E:gooddef}, then we have \eqref{E:deff}.
\end{prop}

\begin{proof}
The functional equation \eqref{E:main} of $\up(t)$ implies for the Fourier transform that 
\[\widehat\up(z)=\frac{\sin\pi z}{\pi z}\widehat\up(z/2).\]
Therefore
\[f(2x)=e^{x}\widehat\up
\Bigl(\frac{ix}{2\pi}\Bigr)=e^{x}\frac{\sin(ix/2)}{ix/2}\widehat\up
\Bigl(\frac{ix}{4\pi}\Bigr)=e^{\frac{x}{2}}\frac{e^{-x/2}-e^{x/2}}{2i(ix/2)}f(x)
=\frac{e^x-1}{x}f(x).\]
\end{proof}

By \eqref{E:gooddef} we have
\[f(x)=1+x\sum_{n=0}^\infty\frac{x^n}{n!}\int_0^1 t^n\up(t)\,dt.\]
So, equating the coefficients of equal powers with the expansion \eqref{E:dexpan}
we obtain 
\cite{MR666839}*{eq.~(42)}
\begin{equation}\label{E:dnint}
d_n=n\int_0^1t^{n-1}\up(t)\,dt.
\end{equation}
Together with \cite{MR666839}*{eq.~(33)} this yields
\begin{equation}\label{E:d-n}
d_n=n!\,2^{\binom{n}{2}}\up(1-2^{-n})=n!\,2^{\binom{n}{2}}F(2^{-n}).
\end{equation}
The values of the $d_n$ can be computed easily applying the recurrence \eqref{E:d-recur}.
The first terms are
\[1,\frac{1}{2},\frac{5}{18},\frac{1}{6},\frac{143}{1\,350},\frac{19}{270},
\frac{1\,153}{23\,814},\frac{583}{17\,010},\frac{1\,616\,353}{65\,063\,250},
\frac{132\,809}{7\,229\,250},
\frac{134\,926\,369}{9\,762\,090\,030},\dots\]

\begin{prop}
The values of  $d_n$ can be computed in terms of  $c_n$:
\begin{equation}\label{E:dncn1}
d_n=\frac{1}{2^n}\sum_{k=0}^{\lfloor n/2\rfloor}\binom{n}{2k}c_k.
\end{equation}
\end{prop}
\begin{proof}
For $n=0$ we have $c_0=d_0=1$ and \eqref{E:dncn1} is trivial in this case. 
For $n\ge1$ we have the following computation
\begin{multline*}
d_n=n\int_0^1 t^{n-1}\up(t)\,dt=\int_0^1\up(t)\,d(t^n)=\Bigl.
t^n\up(t)\Bigr|_{t=0}^{t=1}-\int_0^1t^n\up'(t)\,dt\\
=
-\int_0^12t^n\bigl(\up(2t+1)-\up(2t-1)\bigr)\,dt=2\int_0^1 t^n\up(2t-1)\,dt
=\int_{-1}^1\up(x)\Bigl(\frac{x+1}{2}\Bigr)^n\,dx\\
=
\frac{1}{2^n}\sum_{k=0}^n\int_{-1}^1\binom{n}{k}\up(x)x^k\,dx=
\frac{1}{2^n}\sum_{0\le h\le n/2}\binom{n}{2h}\int\up(x)x^{2h}\,dx\\=
\frac{1}{2^n}\sum_{0\le h\le n/2}\binom{n}{2h}c_h.
\end{multline*}
\end{proof}

As in the case of the $c_n$, there is a sequence of natural numbers hiding in 
$d_n$. 

\begin{prop}
There is a sequence of natural numbers $(G_n)$ such that 
\begin{equation}\label{E:defG}
d_n=\frac{G_n}{(n+1)!}\prod_{k=1}^n(2^k-1)^{-1}.
\end{equation}
\end{prop}

\begin{proof}
For $n\ge0$ define $G_n$ by the equation \eqref{E:defG}. The recurrence \eqref{E:d-recur}
then implies 
\[G_n=\sum_{k=0}^{n-1}G_k\binom{n+1}{k}\frac{n!}{(k+1)!}\prod_{j={k+1}}^{n-1}(2^j-1),\qquad n\ge1.\]
By induction we obtain that all the $G_n$ are natural numbers. 
\end{proof}
The first numbers $G_n$ starting from $n=0$ are 
\[1,\  1,\  5,\  84,\  4\,004,\   494\,760,\   150\,120\,600,\   107\,969\,547\,840,\   179\,605\,731\,622\,464, \dots\]
By \eqref{E:d-n} we obtain the explicit value
\begin{equation}\label{E:values}
F(2^{-n})=\up(1-2^{-n})=\frac{2^{-\binom{n}{2}}G_n}{n!\,(n+1)!}\prod_{k=1}^{n}(2^k-1)^{-1},\qquad n\ge0.
\end{equation} 
Equations \eqref{E:35} and \eqref{E:values} yields
\begin{equation}
\frac{G_{2n+1}}{2n+1}=2^n (n+1)! F_n\prod_{k=0}^n(2^{2k+1}-1),\qquad n\ge0.
\end{equation}
So that $(2n+1)F_n\mid G_{2n+1}$.

\section{Reshetnikov question.}

Recently Vladimir Reshetnikov posed a  question (\href{http://mathoverflow.net/questions/261649/a-conjecture-about-certain-values-of-the-fabius-function}{question 261649
in MathOverflow}) about Fabius function. The results in \cite{MR666839} makes it easy 
to answer. Note that \cite{MR666839} has been translated to English
and posted it in arXiv (\href{http://arxiv.org/abs/1702.05442}{English version arXiv:1702.05442}).
We refer
to this paper for some of the results.

The question of Reshetnikov  is
the following:

\begin{quest}[V. Reshetnikov]
For any natural number $n$ define the numbers
\begin{equation}\label{D:R}
R_n=2^{\binom{n-1}{2}}(2n)!\,F(2^{-n})\prod_{m=1}^{\lfloor n/2\rfloor}(2^{2m}-1).
\end{equation}
It is true that all $R_n$ are natural numbers?
\end{quest}
The first  $R_n$ (starting from $n=1$) certainly are:
\[1,\nn 5,\nn 15,\nn 1\,001,\nn 5\,985,\nn 2\,853\,675,\nn 26\,261\,235,\nn 72\,808\,620\,885,\nn 915\,304\,354\,965,\dots\]
Applying some of the results in \cite{MR666839} we will show that in fact all $R_n$ are
natural numbers. 

In Section \ref{S:comput} we consider the problem of the computation of 
the exact values of the function $F(x)$
at dyadic points. This will allows us to state another conjecture related to this 
conjecture of Reshetnikov.

\section{Answer to Reshetnikov's question.}

Reshetnikov's numbers can be computed in terms of the $d_n$. This  will be 
essential in the solution of Reshetnikov's question.

\begin{prop}\label{P:Rndn}
For all $n\ge1$ we have
\begin{equation}\label{E:Rndn}
R_n=2d_n (2n-1)!! \prod_{k=1}^{\lfloor n/2\rfloor}(2^{2k}-1).
\end{equation}
\end{prop}

\begin{proof}
In the definition \eqref{D:R} of the $R_n$ substitute the value of 
$F(2^{-n})$ obtained from \eqref{E:d-n}.
\end{proof}

First we consider the odd case. 
\begin{thm}\label{T:6}
For all $n\ge0$ the number $R_{2n+1}$ is an integer multiple of $F_n$.
\end{thm}
\begin{proof}
Since $\up$ is an even function, we have by  \eqref{E:dnint}
\[d_{2n+1}=(2n+1)\int_0^1t^{2n}\up(t)\,dt=\frac{2n+1}{2}c_n.\]
Therefore by \eqref{E:Rndn}
\[R_{2n+1}=2d_{2n+1} (4n+1)!! \prod_{k=1}^{n}(2^{2k}-1)=(2n+1)c_n(4n+1)!! \prod_{k=1}^{n}(2^{2k}-1).\]
By \eqref{E:cf}
\[R_{2n+1}=\frac{F_n}{(2n-1)!!}\prod_{\nu=1}^n(2^{2\nu}-1)^{-1}
(4n+1)!! \prod_{k=1}^{n}(2^{2k}-1).\]
Simplifying 
\begin{equation}\label{E:RF}
R_{2n+1}=F_n\frac{(4n+1)!!}{(2n-1)!!}.
\end{equation}
So that $R_{2n+1}$ is a natural number and  $F_n\mid R_{2n+1}$.
\end{proof}

\begin{prop}\label{P:7}
For $n\ge1$ we have
\begin{equation}\label{E:R2n}
R_{2n}=\sum_{k=0}^n \frac{2 F_k}{2^{2n}}\binom{2n}{2k}\frac{(4n-1)!!}{(2k+1)!!}
\prod_{\ell=k+1}^n(2^{2\ell}-1),
\end{equation}
so that the denominator of $R_{2n}$ can be only a power of $2$.
\end{prop}
\begin{proof}
By \eqref{E:dncn1} and \eqref{E:Rndn} we have
\[R_{2n}=
2d_{2n} (4n-1)!! \prod_{\ell=1}^n(2^{2\ell}-1)=
\frac{2}{2^{2n}}\sum_{k=0}^{n}\binom{2n}{2k}c_k (4n-1)!! \prod_{\ell=1}^n(2^{2\ell}-1),\]
so that by \eqref{E:cf}
\[R_{2n}=\sum_{k=0}^n \frac{2}{2^{2n}}\binom{2n}{2k}
(4n-1)!! \frac{F_k}{(2k+1)!!}\prod_{r=1}^k(2^{2r}-1)^{-1}\prod_{\ell=1}^n(2^{2\ell}-1).\]
Simplifying we obtain \eqref{E:R2n}. Since the factor $\frac{(4n-1)!!}{(2k+1)!!}$ is an integer
in \eqref{E:R2n} the  assertion about the denominator follows.
\end{proof}

\begin{thm}\label{T:8}
For all $n$ the number $R_n$ is a natural number.
\end{thm}

\begin{proof}
By Theorem \ref{T:6} it remains only to prove that $R_{2n}$ is a natural number.
For a rational number $r$ let $\nu_2(r)$ be the exponent of $2$ 
in the prime decomposition
of $r$. By Proposition \ref{P:7},  $R_{2n}$ is a 
natural number if and only if $\nu_2(R_{2n})\ge0$. By \eqref{E:Rndn} we have
$\nu_2(R_{n})=\nu_2(2d_{n})$.  We will show by induction that 
$\nu_2(2d_n)\ge0$ and this will prove the Theorem.

Since the numbers $F_n$ are natural, \eqref{E:RF} implies that 
$\nu_2(2d_{2n+1})=\nu_2(R_{2n+1})\ge0$. 
We have $\nu_2(2d_0)=\nu_2(2)=1\ge0$. For any other even number we use induction. By
\eqref{E:d-recur} we have for $n\ge1$
\begin{equation}\label{E:d2n=}
(2n+1)(2^{2n}-1)2d_{2n}=\sum_{k=0}^{2n-1}\binom{2n+1}{k}2d_k.
\end{equation}
Hence 
\[\nu_2(2d_{2n})=\nu_2\bigl((2n+1)(2^{2n}-1)2d_{2n}\bigr)=
\nu_2\Bigl(\sum_{k=0}^{2n-1}\binom{2n+1}{k}2d_k\Bigr),\qquad n\ge1.\]
By induction hypothesis $\nu_2(2d_k)\ge0$ for $0\le k\le 2n-1$. It follows that $\nu_2(2d_{2n})\ge0$
for all $n\ge1$. 
\end{proof}

\section{Computation of $\up(m/2^n)$.}\label{S:comput}

At dyadic points $t=m/2^n$  the function takes rational values. This was proved 
in \cite{MR666839}*{Th.~9}, where a procedure to compute these values was given.
In \cite{MR666839} it was desired to express $\up(q/2^n)$ in terms of $\up(1-2^{-k-1})$. 
This was achieved by substituting $c_k$ in terms of $\up(1-2^{-k-1})$. This is 
correct, but for computations it is preferable not to make this last substitution. 
In this way the formula in \cite{MR666839}*{Thm.~7} can be written
\[\up(q2^{-n})=\frac{2^{-\binom{n+1}{2}}}{n!}\sum_{h=0}^{q+2^n-1}(-1)^{w(h)}
\sum_{k=0}^{\lfloor n/2\rfloor}\binom{n}{2k}\bigl(2(q-h)+2^{n+1}-1\bigr)^{n-2k}c_k.\]
This is valid for $-2^{-n}\le q\le 2^n$. 
The sum in $h$ has few terms for $q$ near $-2^{n}$.
Therefore since $\up(1-x)=\up(-1+x)$ we obtain, with $q=a-2^{n}$,
\begin{equation}\label{E:explicit}
F(\tfrac{a}{2^n})=\up(1-\tfrac{a}{2^n})=
\frac{2^{-\binom{n+1}{2}}}{n!}\sum_{h=0}^{a-1}(-1)^{w(h)}
\sum_{k=0}^{\lfloor n/2\rfloor}\binom{n}{2k}(2a-2h-1)^{n-2k}c_k.
\end{equation}
The numbers $c_k$ can be computed with the recurrence \eqref{E:cr2}. 
This  equation, together with $\up(t)+\up(1-t)=1$ for $0\le t\le1$,  allows one to easily 
compute the values at dyadic points. 
Another way has been proposed by Haugland recently \cite{H}.
Notice also that our formula can be written as follows
\begin{equation}\label{E:explicit2}
\up(1-\tfrac{a}{2^n})=
\frac{2^{-\binom{n+1}{2}}}{n!}
\sum_{k=0}^{\lfloor n/2\rfloor}\binom{n}{2k}c_k
\sum_{h=0}^{a-1}(-1)^{w(h)}
(2a-2h-1)^{n-2k}.
\end{equation}

Equation \eqref{E:explicit2} show that the values at dyadic points of $\up(x)$, and therefore also those  
of $F(x)$, are rational numbers. In principle it allow the explicit computation of $\up(a/2^n)$ 
and therefore of $F(a/2^n)$ for all integers $a$. Nevertheless there is a better way to compute the function.
It is thought that  this better method is due to V. A. Rvach{\"e}v,  
brought to attention by V. Reshetnikov. He posted a solution to a 
\href{http://mathematica.stackexchange.com/questions/120331/how-do-i-numerically-evaluate-and-plot-the-fabius-function}{question of Pierrot Bolnez} about how to compute $F(x)$. 
Reshetnikov's answer consists of a program 
in Mathematica to compute the function. The program has been reverse engineered, 
and proved the 
relevant result.  At other point Reshetnikov cites Rvach{\"e}v and  says that he thinks it is not 
translated to English.  So a proof of the formula on which the method is based can be useful.

We need a Lemma.
\begin{lem}\label{L:taylor}
Let $x>0$ be a real number and let  $k$ be the unique integer such that $2^k\le x<2^{k+1}$. Then for any 
non negative integer $n$ we have
\begin{equation}
\int_{2^k}^x(x-t)^nF^{(n+1)}(t)\,dt=-\int_0^{y}(y-t)^n F^{(n+1)}(t)\,dt,\qquad \text{with}\quad y=x-2^k.
\end{equation}
\end{lem}

\begin{proof}
Putting $t=2^k(1+u)$ in the integral we obtain
\[\int_{2^k}^x(x-t)^nF^{(n+1)}(t)\,dt=2^k\int_{0}^{2^{-k}y}(y-2^k u)^nF^{(n+1)}(2^k(1+u))\,du.\]
Letting $a=2^{-k}y$, we have $0\le a<1$  and we obtain
\[\int_{2^k}^x(x-t)^nF^{(n+1)}(t)\,dt=2^{k(n+1)}\int_{0}^{a}(a-u)^nF^{(n+1)}(2^k(1+u))\,du.\]
Differentiating $F^{(k)}(x)=2^{\binom{k+1}{2}}F(2^k t)$ repeatedly we obtain
\[F^{(k+n+1)}(t)=2^{\binom{k+1}{2}} 2^{k(n+1)}F^{(n+1)}(2^k t).\]
It follows that 
\[2^{k(n+1)}\int_{0}^{a}(a-u)^nF^{(n+1)}(2^k(1+u))\,du=2^{-\binom{k+1}{2}}\int_{0}^{a}(a-u)^nF^{(k+n+1)}(u+1)\,du.\]
When $0<u<a$ we have $1<u+1<2$. For $0<x<2$ we have by definition $F(x)=\up(x-1)$. 
Therefore $F'(x)=2\up(2x-1)-2\up(2x-3)$. This implies for $0<u<1$ that $F'(u)=2\up(2u-1)$ and $F'(u+1)=-2\up(2u-1)=-F'(u)$
Hence
\[2^{-\binom{k+1}{2}}\int_{0}^{a}(a-u)^nF^{(k+n+1)}(u+1)\,du=-2^{-\binom{k+1}{2}}\int_{0}^{a}(a-u)^nF^{(k+n+1)}(u)\,du\]
It only remains to reverse our steps.  We have
\begin{multline*}
-2^{-\binom{k+1}{2}}\int_{0}^{a}(a-u)^nF^{(k+n+1)}(u)\,du=
-\int_{0}^{a}(a-u)^n 2^{k(n+1)}F^{(n+1)}(2^k u)\,du
\\
-\int_{0}^{y}(y-v)^n F^{(n+1)}(v)\,dv
\end{multline*}

\end{proof}

\begin{prop}\label{P:Res}
Let $x>0$ be a real number. Let  $n$ be the unique integer  such that $2^{-n}\le x< 2^{-n+1}$ and put $y=x-2^{-n}$.
Then we have 
\begin{equation}
F(x)=-F(y)+\sum_{0\le k\le n} 2^{\binom{k+1}{2}-\binom{n-k}{2}}\frac{d_{n-k}}{(n-k)!}\frac{y^k}{k!},
\end{equation}
where the sum is $0$ when $n<0$. 
\end{prop}

\begin{proof}
We use Taylor's Theorem with integral rest 
\[F(x)=F(y+2^{-n})=\sum_{0\le k\le N}\frac{F^{(k)}(2^{-n})}{k!}y^k+\int_{2^{-n}}^x\frac{(x-t)^N}{N!} F^{(N+1)}(t)\,dt.\]

When $n<0$, the number $2^{-n}$ is an even natural number. For $k\ge0$ the derivative $F^{(k)}(2^{-n})=0$. So that for 
any $N\ge0$, 
\[F(x)=\int_{2^{-n}}^x\frac{(x-t)^N}{N!} F^{(N+1)}(t)\,dt\]
By Lemma \eqref{L:taylor},
\[\int_{2^{-n}}^x\frac{(x-t)^N}{N!} F^{(N+1)}(t)\,dt=-\int_{0}^y\frac{(x-t)^N}{N!} F^{(N+1)}(t)\,dt.\]
Again, by Taylor's Theorem, applied at the point $0$ where all derivatives are also $=0$,  
this is equal to $-F(y)$ as we wanted to prove. 

For $n\ge0$ we apply again Taylor's Theorem with $N=n$.  In this case
\[F(x)=F(y+2^{-n})=\sum_{0\le k\le n}\frac{F^{(k)}(2^{-n})}{k!}y^k+\int_{2^{-n}}^x\frac{(x-t)^n}{n!} F^{(n+1)}(t)\,dt.\]
Lemma \eqref{L:taylor} gives us 
\[\int_{2^{-n}}^x\frac{(x-t)^n}{n!} F^{(n+1)}(t)\,dt=-\int_0^y \frac{(y-t)^n}{n!} F^{(n+1)}(t)\,dt=-F(y).\]
The last equality using Taylor's Theorem at the point $t=0$ where all derivatives of $F$ are equal to $0$. 

To finish the proof notice that by \eqref{E:FabiusD} and \eqref{E:d-n} we have for $0\le k\le n$,
\[F^{(k)}(2^{-n})=2^{\binom{k+1}{2}}F(2^k 2^{-n})=2^{\binom{k+1}{2}}\frac{d_{n-k}}{(n-k)! 2^{\binom{n-k}{2}}}.\]
\end{proof}

\begin{rem}
Notice that  Proposition \ref{P:Res} reduces the computation of $F(x)$  to that of $F(y)$. 
If $x$ is dyadic a repeated use ends in $F(0)$. The number of steps being merely the number of $1$ 
in the  dyadic representation of $x$.  If $x$ is not dyadic a finite number of steps reduce the computation to that 
of  $F(y)$ for some $y$ small enough so that  $F(y)$ is less than the absolute error requested. 
This procedure is implemented by Reshetnikov in his answer to \href{http://mathematica.stackexchange.com/questions/120331/how-do-i-numerically-evaluate-and-plot-the-fabius-function}{Pierrot Bolnez question in Mathematica StackExchange}.
\end{rem}

\section{The least common denominator of $\up(a/2^n)$.}

Reshetnikov's conjecture tries to determine the denominator of the dyadic number $F(2^{-n})$.  It is more natural to 
consider the common denominator $D_n$ of all numbers $F(a/2^n)$ where $a$ is an odd number. This is a finite number
because we only have to consider the case of $a=2k+1$ for $k=0$, $1$, \dots, $2^{n-1}-1$.  

Therefore we define as follow.
\begin{defn} Let $D_n$ be le least common multiple of the denominators of the 
fractions $\up(\frac{2k+1}{2^n})$ with $2^{-n}<2k+1\le2^n$. 
\end{defn}

The first $D_n$ starting from $D_0=1$ are 
\begin{gather*}
1,\nn 2,\nn 72,\nn 288,\nn 2\,073\,600,\nn 33\,177\,600,\nn 2\,809\,213\,747\,200,\nn 179\,789\,679\,820\,800,\\
704\,200\,217\,922\,109\,440\,000,\nn 180\,275\,255\,788\,060\,016\,640\,000,\\
6\,231\,974\,256\,792\,696\,936\,191\,754\,240\,000,\nn 6\,381\,541\,638\,955\,721\,662\,660\,356\,341\,760\,000,\\
292214732887898713986916575925267070976000000,\\
1196911545908833132490410294989893922717696000000,\\
963821659256803158077786940841300728342971034894336000000,\dots
\end{gather*}

\begin{thm}\label{T:9}
For any natural number $n\ge1$ and any integer $-2^n< a< 2^n$ the number
\begin{equation}\label{Rekete}
\up(\tfrac{a}{2^n})\;n!\;2^{\binom{n+1}{2}}(2\lfloor n/2\rfloor+1)!! \prod_{k=1}^{\lfloor n/2\rfloor}(2^{2k}-1)
\end{equation}
is natural. 

Equivalently 
\[D_n\text{ divides the number } n!\;2^{\binom{n+1}{2}}(2\lfloor n/2\rfloor+1)!! \prod_{k=1}^{\lfloor n/2\rfloor}(2^{2k}-1).\]

\end{thm}

\begin{proof}
By \eqref{E:explicit} we know that 
\[\up(\tfrac{a}{2^n})\;n!\;2^{\binom{n+1}{2}}\]
is a linear combination with integer coefficients of the rational numbers $c_k$ with 
$0\le k\le n/2$. We have $c_0=1$ and, for $n\ge1$, by \eqref{E:cf} the numbers 
\[F_k = c_k(2k+1)!! \prod_{r=1}^{k}(2^{2r}-1)\]
are natural. Therefore, if $m=\lfloor n/2\rfloor$ then all the numbers
\[c_k (2m+1)!! \prod_{r=1}^{m}(2^{2r}-1), \qquad 0\le k\le n/2,\]
are natural.  It follows that the number in \eqref{Rekete} is an integer. Since 
it is also positive it is a natural number.
\end{proof}
\begin{rem}
Theorem \ref{T:9} applies to any $a$ and not only for $2^n-1$ as in Reshetnikov's conjecture.
The power of $2$ is better in Reshetnikov's conjecture than in Theorem \ref{T:9}, but 
for other primes  Theorem \ref{T:9} is better than Reshetnikov's conjecture.  It is thought that 
the power of $2$ can be improved in Theorem \ref{T:9}. 
\end{rem}

\begin{prop}
For any natural number $n$ the number $D_n$ is a multiple of the denominator of 
$d_n/(2^{\binom{n}{2}}\,n!)$.
\end{prop}
\begin{proof}
This follows directly from equation \eqref{E:d-n} and the definition of 
$D_n$, since $D_n$ is a multiple of the denominator of $\up(1-2^{-n})$.
\end{proof}

The quotients $D_n/\den(d_n/(2^{\binom{n}{2}}\,n!))$ for $1\le n\le 17$ are
\[1,\nn 1,\nn 1,\nn 1,\nn 1,\nn 5,\nn 1,\nn 1,\nn 1,\nn 5,\nn 1,\nn 1,\nn 1,\nn 55,\nn1,\nn13,\nn11,\nn\dots\]

\begin{conj}\label{C:nueva}
The following appear to be true:
\begin{itemize}
\item[(a)]
The quotient $A_n=2^{-\binom{n}{2}}D_n$ satisfies $A_{2n}=A_{2n+1}$ for $n\ge1$.
\item[(b)] Let $K_{n}=A_{2n-1}$ for $n\ge1$, so that 
by (a) we have $K_1=A_1$, and for $n\ge2$
\[K_n=2^{-\binom{2n-1}{2}}D_{2n-1}=2^{-\binom{2n-2}{2}}D_{2n-2}.\]
Then $2\times(2n-1)!$ divides $K_n$.
\item[(c)] The quotients $H_n=\frac{K_n}{2(2n-1)!}$ are odd integers.
\end{itemize}
\end{conj}

The first few $H_n$ starting from $H_1=1$ are
\begin{gather*}
1,\nn 3,\nn 135,\nn 8\,505,\nn 3\,614\,625,\nn 2\,218\,656\,825,\nn 317\,988\,989\,443\,125,\\
148\,846\,103\,258\,477\,625,\nn 8\,607\,025\,920\,921\,468\,665\,625,\nn \dots
\end{gather*}

If these conjectures are true we obtain for $n\ge1$
\[D_{2n-1}=2^{1+\binom{2n-1}{2}}(2n-1)!\, H_n,\quad D_{2n}=
2^{1+\binom{2n}{2}}(2n+1)!\, H_{n+1}.\]

\section{On the nature of the numbers $d_n$.}

In this Section we show that the numerator and denominator of $2d_n$ are odd numbers.

We have noticed in the proof 
of Theorem \ref{T:8} that $\nu_2(2d_n)=\nu_2(R_n)$. So that by equation
\eqref{E:RF} it follows that $\nu_2(2d_{2n+1})=\nu_2(F_n)$. We will show
now that the numbers $F_n$ are odd.  
This will prove our result for $2d_n$ when $n$ is an odd number.
For this proof we will use Lucas's Theorem \cite{MR0023257}.

\begin{thm}[Lucas]
Let $p$ be a prime and $n$ and $k$ two positive integers that we write as
\[n=n_0+n_1p+n_2p^2+\cdots+n_a p^a,\quad k=k_0+k_1p+k_2p^2+\cdots+k_a p^a, 0\le n_j,k_j<p\]
in base $p$.  Then we have the congruence
\[\binom{n}{k}\equiv\binom{n_a}{k_a}\cdots\binom{n_2}{k_2}\binom{n_1}{k_1}\binom{n_0}{k_0}
\pmod{p}.\]
\end{thm}

From Lucas's Theorem we get the following application:
\begin{prop}\label{P:afterLucas}
(a) For $n\ge1$ the cardinal of odd combinatory numbers $\binom{2n+1}{2k}$ 
with $0\le k\le n$ is exactly $2^{w(n)}$.

(b) For $n\ge 1$ the number of odd combinatory numbers $\binom{2n+1}{k}$ 
with $0\le k\le 2n+1$ is  $2^{w(n)+1}$.

\end{prop}

\begin{proof}
Let the binary expansion of $n$ be $\varepsilon_r\varepsilon_{r-1}\dots\varepsilon_0$, that is, 
\[n=\varepsilon_r2^r+\varepsilon2^{r-1}+\cdots \varepsilon_0,\quad \text{with }
\varepsilon_j\in\{0,1\}.\]
The expansion of $2n+1$ is $\varepsilon_r\varepsilon_{r-1}\dots\varepsilon_01$.
If $\binom{2n+1}{2k}$ is odd $k\le n$, then $k\le 2^{r+1}$. Hence its 
binary expansion is
of the form $k=\delta_r\cdots \delta_0$. Therefore Lucas's Theorem says that
\[\binom{2n+1}{2k}\equiv \binom{\varepsilon_r}{\delta_r}\cdots
\binom{\varepsilon_0}{\delta_0}\binom{1}{0}\pmod{2}\]
Here each factor $\binom{0}{1}=0$ and each factor $\binom{1}{\delta}=1$ for
any value of $\delta$.  So $\binom{2n+1}{2k}\equiv1$ is equivalent to saying that
$ \delta_j=0$ when $\varepsilon_j=0$. Therefore the number of odd combinatory 
numbers is $2^{m}$, where $m$ is the number of digits equal to $1$ in 
the expansion of $n$. But 
this is to say $m=w(n)$. 

The second assertion can be proved in the same way.
\end{proof}

\begin{prop}
All the numbers $F_n$  are odd. 
\end{prop}

\begin{proof}
We have $F_0=F_1=1$. We proceed by induction assuming that we have proved that
all $F_k$ are odd numbers for $0\le k\le n-1$. 
Equation \eqref{Rec:F} gives $F_n$ as a sum of terms. Each term is a product 
of an odd number for a combinatory number $\binom{2n+1}{2k}$, with $k$ running on 
$0\le k\le n-1$. The missing combinatory number 
(that is, for $k=n$ is $\binom{2n+1}{2n}=2n+1$) is an 
odd number. So, by Proposition \ref{P:afterLucas}, the term adding to $F_n$ are even except for $2^{w(n)}-1$ of them. Since 
$w(n)\ge1$ this implies that $F_n$ is odd.
\end{proof}

\begin{thm}
For all $n\ge1$ we have $\nu_2(2d_n)=0$, that is, the rational number $2d_n$ is the 
quotient of two odd numbers.
\end{thm}

\begin{proof}
We have seen that $F_n$ is an odd number and that this implies $\nu_2(2d_{2n+1})=0$.
So we only need to show this for $2d_{2n}$ and $n\ge1$ (notice that $2d_0=2$, so that
$\nu_2(2d_0)=1$). 

We have shown (Theorem \ref{T:8}) that $R_n$ are integers. By Proposition \ref{P:Rndn}
there is an odd number
\[N_n=(2n-1)!!\prod_{k=1}^{\lfloor n/2\rfloor}(2^{2k}-1),\]
such that $2d_n N_n=R_n$ is an integer. It is easy to see that for $k\le n$
we have $N_k\mid N_n$.  Therefore for all $1\le k\le n$ we have $d_kN_n$
is an integer. 

We proceed by induction, assuming we have proved that $\nu(2d_k)=0$ for $1\le k<2n$
and we have to show that $\nu_2(2d_{2n})=0$.
We multiply \eqref{E:d2n=} for $n\ge1$  by $N_{2n}$
and obtain
\[\nu_2(2d_{2n})=\nu_2\bigl((2n+1)(2^{2n}-1)2d_{2n}N_{2n}\bigr)
=\nu_2\Bigl(\sum_{k=0}^{2n-1}\binom{2n+1}{k}2d_kN_{2n}\Bigr)\]
By the induction hypothesis $2d_kN_{2n}$ is an odd natural number for $1\le k\le 2n-1$
and it is even for $k=0$. So we only need to show that there are an odd number of
odd combinatory numbers $\binom{2n+1}{k}$ with $1\le k\le 2n-1$. This follows easily 
from Propostition \ref{P:afterLucas}.
\end{proof}

\begin{thm}
The number $R_n$ is odd for $n\ge1$. Or equivalently 
\begin{equation}
\nu_2\bigl(\up(1-2^{-n})\bigr)=-\binom{n}{2}-1-\nu_2(n!).
\end{equation}
\end{thm}
\begin{proof}
We have seen that $2d_n$ is a quotient of two odd numbers and $R_n$ is a 
natural number.   Therefore equation \eqref{E:Rndn} implies
that $R_n$ is an odd number. 

Equation \eqref{E:d-n} yields
\[0=\nu_2(2d_n)=1+\nu_2(n!)+\binom{n}{2}+\nu_2\bigl(\up(1-2^{-n})\bigr).\]
\end{proof}

\section{Relation with Bernoulli numbers.}

The numbers $c_n$ and $d_n$ are related to Bernoulli numbers. Perhaps this
may be useful to solve  Conjecture \ref{C:nueva}.

\begin{prop}
The numbers $d_n$ and $c_n$ are given by the recurrences
\begin{equation}\label{Bncn}
c_0=1,\quad c_n=\frac{1}{2^{2n}-1}\sum_{k=1}^n 2^{2n-2k}(2^{2k}-2)\binom{2n}{2k}
B_{2k}c_{n-k},
\end{equation}
and
\begin{equation}\label{Bndn}
d_0=1,\quad d_n=\frac{n 2^{n-2}}{2^n-1} d_{n-1}-\frac{1}{2^n-1}
\sum_{k=1}^{\lfloor n/2\rfloor}\binom{n}{2k}2^{n-2k}B_{2k} d_{n-2k},
\end{equation}
where $B_k$ are Bernoulli numbers $B_0=1$, $B_1=-\frac12$, $B_2=\frac16$, $B_3=0$,  \dots
\end{prop}

\begin{proof}
By \eqref{hatphi} and \eqref{E:deff} the functions
\[\widehat\up(z)=\sum_{n=0}^\infty (-1)^n \frac{c_n}{(2n)!} (2\pi z)^{2n},\quad
\text{and}\quad f(x)=\sum_{n=0}^\infty \frac{d_n}{n!}x^n\]
satisfy the equations
\[\frac{\pi z}{\sin\pi z}\widehat\up(z)=\widehat\up(\tfrac{z}{2}),\quad
\frac{x}{e^x-1}f(2x)=f(x).\]
Using the well known power series
\[\frac{\pi z}{\sin\pi z}=1+\sum_{n=1}^\infty(-1)^{n+1}\frac{B_{2n}(2^{2n}-2)}{(2n)!}(\pi z)^{2n},\quad \frac{x}{e^x-1}=\sum_{n=0}^\infty \frac{B_n}{n!} x^n\]
and equating coefficients of the same power of the variable we obtain our 
recurrences.
\end{proof}


\begin{bibdiv}
\begin{biblist}


\bib{MR666839}{article}{
   author={Arias de Reyna, J.},
   title={Definici\'on y estudio de una funci\'on indefinidamente diferenciable de soporte compacto},
   language={Spanish},
   journal={Rev. Real Acad. Cienc. Exact. F\'\i s. Natur. Madrid},
   volume={76},
   date={1982},
   number={1},
   pages={21--38},
   issn={0034-0596},
   review={\MR{666839}},
   note={{\href{http://arxiv.org/abs/1702.05442}{English version arXiv:1702.05442}}}
}

	
\bib{MR0197656}{article}{
   author={Fabius, J.},
   title={A probabilistic example of a nowhere analytic $C^{\infty
   }$-function},
   journal={Z. Wahrscheinlichkeitstheorie und Verw. Gebiete},
   volume={5},
   date={1966},
   pages={173--174 (1966)},
   review={\MR{0197656}},
   doi={10.1007/BF00536652},
}



\bib{MR0023257}{article}{
   author={Fine, N. J.},
   title={Binomial coefficients modulo a prime},
   journal={Amer. Math. Monthly},
   volume={54},
   date={1947},
   pages={589--592},
   issn={0002-9890},
   review={\MR{0023257}},
   doi={10.2307/2304500},
}
		

\bib{H}{article}{
   author={Haugland, J. K.},
   title={Evaluating the Fabius function},
   journal={arXiv:1609.07999},
   date={2016},
}


\bib{MR1501802}{article}{
   author={Jessen, B.},
   author={Wintner, A.},
   title={Distribution functions and the Riemann zeta function},
   journal={Trans. Amer. Math. Soc.},
   volume={38},
   date={1935},
   pages={48--88},
}

\bib{MR705073}{article}{
   author={Kirov, G. Kh.},
   author={Totkov, G. A.},
   title={Distribution of zeros of derivatives of the function $\lambda
   (x)$},
   language={Bulgarian, with English and Russian summaries},
   journal={Godishnik Vissh. Uchebn. Zaved. Prilozhna Mat.},
   volume={17},
   date={1981},
   number={4},
   pages={157--165},
   review={\MR{705073}},
}

\bib{MR1350871}{article}{
   author={Kolodyazhny, V. M.},
   author={Rvach{\"e}v, V. A.},
   title={On some elementary orthogonal wavelet systems},
   language={Ukrainian},
   journal={Dopov. Nats. Akad. Nauk Ukra\"\i ni},
   date={1995},
   number={2},
   pages={20--22},
   issn={0868-8044},
   review={\MR{1350871}},
}
		
\bib{MR1294281}{article}{
   author={Kolodiazhny, V. M.},
   author={Rvach{\"e}v, V. A.},
   title={On the construction of Y. Meyer wavelets using ${\rm up}(x)$
   function},
   language={English, with Russian and Ukrainian summaries},
   journal={Dopov./Dokl. Akad. Nauk Ukra\"\i ni},
   date={1993},
   number={10},
   pages={18--24},
   issn={0868-8044},
   review={\MR{1294281}},
}

\bib{MR0296237}{article}{
   author={Rvachev, V. L.},
   author={Rvach{\"e}v, V. A.},
   title={A certain finite function},
   language={Ukrainian, with English and Russian summaries},
   journal={Dopov\=\i d\=\i\ Akad. Nauk Ukra\"\i n. RSR Ser. A},
   date={1971},
   pages={705--707, 764},
   issn={0201-8446},
   review={\MR{0296237}},
}



\bib{MR0330388}{article}{
   author={Rvach{\"e}v, V. A.},
   title={Certain functions with compact support, and their applications},
   language={Russian},
   conference={
      title={Mathematical physics, No. 13 (Russian)},
   },
   book={
      publisher={``Naukova Dumka'', Kiev},
   },
   date={1973},
   pages={139--149, 194},
   review={\MR{0330388}},
}
		

\bib{MR550591}{article}{
   author={Rvach{\"e}v, V. A.},
   title={An orthonormal system on the basis of the function $up(x)$},
   language={Russian},
   conference={
      title={Mathematical analysis and probability theory (Russian)},
   },
   book={
      publisher={``Naukova Dumka'', Kiev},
   },
   date={1978},
   pages={146--150, 220},
   review={\MR{550591}},
}

\bib{MR0447908}{article}{
   author={Rvach{\"e}v, V. A.},
   title={Approximation by means of the function ${\rm up}(x)$},
   language={Russian},
   journal={Dokl. Akad. Nauk SSSR},
   volume={233},
   date={1977},
   number={2},
   pages={295--296},
   issn={0002-3264},
   review={\MR{0447908}},
}



\bib{MR1050928}{article}{
   author={Rvach{\"e}v, V. A.},
   title={Compactly-supported solutions of functional-differential equations
   and their applications},
   language={Russian},
   journal={Uspekhi Mat. Nauk},
   volume={45},
   date={1990},
   number={1(271)},
   pages={77--103, 222},
   issn={0042-1316},
   translation={
      journal={Russian Math. Surveys},
      volume={45},
      date={1990},
      number={1},
      pages={87--120},
      issn={0036-0279},
   },
   review={\MR{1050928}},
}

\bib{MR1031820}{article}{
   author={Schnabl, R.},
   title={\"Uber eine $C^\infty$-Funktion},
   language={German, with English summary},
   conference={
      title={Zahlentheoretische Analysis},
   },
   book={
      series={Lecture Notes in Math.},
      volume={1114},
      publisher={Springer, Berlin},
   },
   date={1985},
   pages={134--142},
   review={\MR{1031820}},
}


\bib{MR1271144}{book}{
   author={Stromberg, Karl R.},
   title={Probability for analysts},
   series={Chapman \& Hall Probability Series},
   note={Lecture notes prepared by Kuppusamy Ravindran},
   publisher={Chapman \& Hall, New York},
   date={1994},
   pages={xiv+330},
   isbn={0-412-04171-5},
   review={\MR{1271144}},
   doi={10.1672/08-124.1},
}


\end{biblist}
\end{bibdiv}
\end{document}